\documentclass[a4paper,twoside,reqno,10pt]{amsart}
\usepackage{amssymb, amsmath}

\usepackage{color}
\usepackage{enumerate}
\usepackage{xy}
\xyoption{all}

\newcommand{\received}{Received on 19 December 2008.}

\newcommand{\theaddress}{David Lewis and Thomas Unger,\\
                        School of Mathematical Sciences,\\
                        University College Dublin,\\\
                        Belfield,\
                        Dublin~4, Ireland,\\
                        \ninesl david.lewis@ucd.ie, thomas.unger@ucd.ie}

\newfont{\caps}{cmcsc10}
\newfont{\smallcaps}{cmcsc10 at 8pt}
\newfont{\ninerm}{cmr9}
\newfont{\nineit}{cmti9}
\newfont{\ninesl}{cmsl9}
\newfont{\eightrm}{cmr8}
\newfont{\eightit}{cmti8}
\newfont{\eightsl}{cmsl8}
\newfont{\eightbf}{cmbx8}
\newfont{\sevenrm}{cmr7}
\newfont{\sevenit}{cmti7}

\newcommand{\too}{\longrightarrow}
\newcommand{\mapstoo}{\longmapsto}

\renewcommand{\bar}{\overline}

\newcommand{\bbar}{-}

\DeclareMathOperator{\End}{End}

\newcommand{\x}{\times}

\DeclareMathOperator{\ad}{ad}

\newcommand{\ve}{\varepsilon}
\newcommand{\vf}{\varphi}

\newcommand{\Gr}{\mathrm{Gr}}

\newtheorem{lemma}{Lemma}[section]

\newtheorem{propo}[lemma]{Proposition}

\theoremstyle{definition}

\newtheorem{remark}[lemma]{Remark}

\DeclareFontFamily{OT1}{manual}{}
\DeclareFontShape{OT1}{manual}{m}{n}{ <10> manfnt }{}

\begin{document}
\title{}
\begin{center}
\ \ \
\end{center}
\vspace{1truecm}
\begin{center}
\bf\large Hermitian Morita Theory:\\ a Matrix Approach
\end{center}
\vspace{-1truecm}
\author{David W. Lewis}
\author{Thomas Unger}

\begin{abstract} 
In this  note an explicit matrix description of hermitian Morita theory is presented.
\end{abstract}

\maketitle

\thispagestyle{myheadings} \markboth{\smallcaps%
                      David W. Lewis and Thomas Unger}
         {{\eightit Irish Math.~Soc.~Bulletin {\eightbf62} (2008),
           37--41}\hfill\eightrm}

\section{Introduction}

\noindent
Let $K$ be a field of characteristic different from two and let $A$ be a central simple $K$-algebra equipped with an involution~$*$. 
By a well-known theorem of Wedderburn, $A$ is of the form $M_n(D)$, a full matrix algebra over a division $K$-algebra $D$. 
Furthermore, there exists an involution $\bbar$ on $D$ of the same kind as $*$ such that $*$ and $\bbar$ have the same restriction to~$K$.
Then $*$ is the adjoint involution $\ad_{h_0}$ of some nonsingular $\ve_0$-hermitian form $h_0$ over $(D,\bbar)$,
\[h_0: D^n\x D^n \too D,\]
with $\ve_0=\pm 1$. Thus
\[X^*=\ad_{h_0}(X)=S \bar{X}^t S^{-1},\quad \forall X\in M_n(D),\]
where $S\in GL_n(D)$ is the matrix of $h_0$, so that $\bar{S}^t=\ve_0 S$.

Let $\Gr_\ve(A,*)$ and $W_\ve(A,*)$ denote the Grothendieck group and Witt group of  $\ve$-hermitian forms over $(A,*)$, respectively. 
Hermitian Morita theory furnishes us with  isomorphisms
\[\Gr_{\ve}(A,*) \cong \Gr_{\ve_0 \ve} (D,\bbar)\quad\text{and}\quad W_{\ve}(A,*) \cong W_{\ve_0 \ve} (D,\bbar).\]
These isomorphisms are the result of the following equivalences of categories
\[\xymatrix@1{
\left\{\!\parbox[c]{5.03em}{\flushleft$\ve$-hermitian forms over $(M_n(D), *)$}\!\right\} \ar@{<->}[r]^-{\text{scaling}}&  
\left\{\!\parbox[c]{5.94em}{\flushleft$\ve_0\ve$-hermitian forms over $(M_n(D), \bbar^t)$}\!\right\}\ar@{<->}[rr]^-{\text{Morita}}_-{\text{equivalence}} & &
\left\{\!\parbox[c]{5.94em}{\flushleft $\ve_0\ve$-hermitian forms over $(D, \bbar)$}\!\right\}
}\]
(all forms are assumed to be nonsingular)
which respect isometries, orthogonal sums and hyperbolic forms.

In this note we describe these correspondences explicitly. In particular we give a matrix description of Morita equivalence which does 
not seem to be generally known. Other explicit descriptions can be found in \cite{Lew, Ri1, Ri2}. The subject is often treated in a more 
abstract manner, such as in~\cite{FM} and \cite[Chap.~I, \S9]{Knus}.

\markboth{\smallcaps David W. Lewis and Thomas Unger}
{\smallcaps Hermitian Morita Theory: a Matrix Approach}

\section{Scaling}

\noindent
Let $M$ be a right $M_n(D)$-module and let $h:M\x M\too M_n(D)$ be an $\ve$-hermitian form with respect to $*$, i.e.
\[h(y,x)=\ve h(x,y)^*=\ve S \bar{h(x,y)}^t S^{-1}.\]

\begin{propo} The form
\[S^{-1}h: M\x M\too M_n(D),\ (x,y)\mapstoo S^{-1} h(x,y)\]
is $\ve_0\ve$-hermitian over $(M_n(D), \bbar^t)$.
\end{propo}
\begin{proof}  Sesquilinearity of $S^{-1}h$ with respect to $\bbar^t$ follows easily from sesquilinearity of $h$ with respect to $*$:
\begin{equation*}
\begin{split}
(S^{-1}h)(x\alpha,y) &=S^{-1}h(x\alpha,y)=S^{-1} \alpha^* h(x,y)\\
                                         &=S^{-1} S\bar{\alpha}^tS^{-1}h(x,y)=\bar{\alpha}^tS^{-1} h(x,y)
\end{split}
\end{equation*}
for any $\alpha \in M_n(D)$ and any $x,y\in M$.

Furthermore, using the fact that $\bar{S}^t=\ve_0 S$, we get
\begin{align*}
(S^{-1}h)(y,x) &=S^{-1} h(y,x)\\
&= S^{-1} \ve S \bar{h(x,y)}^t S^{-1}\\
&=  \ve  \bar{h(x,y)}^t S^{-1}\\
&=\ve \ve_0 \bar{h(x,y)}^t \bar{(S^{-1} )}^t\\
&=\ve \ve_0 \bar{(S^{-1}h)(x,y)}^t
\end{align*}
for any $x,y\in M$.
\end{proof}

\begin{remark} By the first part of the proof, scaling of a sesquilinear form $h$  (rather than an $\ve$-hermitian form $h$) 
with respect to $*$ results in a sesquilinear form $S^{-1}h$ with respect to $\bbar^t$.
\end{remark}

\begin{remark} The matrix $S$ is not determined uniquely, but only up to  scalar multiplication by $\lambda \in K$, since 
$\lambda S$ and $S$ give the same involution $\ad_{h_0}$. Hence the scaling correspondence is not canonical.
\end{remark}

\section{Morita Equivalence}

\noindent
Every module over $M_n(D)\cong \End_D (D^n)$ is a direct sum of simple modules, namely copies of $D^n$. 
Let $(D^n)^k$ be such a module. We identify $(D^n)^k$ with $D^{k\x n}$, the $k\x n$-matrices over $D$. We view each row of a 
$k\x n$-matrix over $D$ as an element of $D^n$. Note that $M_n(D)$ acts on $D^{k\x n}$ on the right.

Now let
\[h:D^{k\x n} \x D^{k\x n}\too M_n(D)\]
be an $\ve$-hermitian form over $(M_n(D), \bbar^t)$.

\begin{propo}\label{prop2} There exists an $\ve$-hermitian $k\x k$-matrix $B\in M_k(D)$ such that
\begin{equation}\label{eq2}
h(x,y)=\bar{x}^t By,\ \forall x,y\in D^{k\x n}.
\end{equation}
\end{propo}

\begin{proof} Let $B=(b_{ij})$. We will determine the entries $b_{ij}$. Let $e_{ij}\in D^{k\x n}$,  $e'_{ij}\in D^{n\x k}$ and 
$E_{ij}\in M_n(D)$ respectively denote the $k\x n$-matrix, the $n\x k$-matrix and the $n\x n$-matrix with $1$ in the $(i,j)$-th position and 
zeroes everywhere else. One can easily verify that
\begin{equation}\label{eq1}
e_{if}E_{f\ell}=e_{i\ell},
\end{equation}
where $1\leq i\leq k$ and $1\leq f,\ell\leq n$. Also note that if $C\in M_n(D)$, then computing the product $E_{ij}C$ picks the $j$-th row 
of $C$ and puts it in row $i$ while making all  other entries zero. Similarly,
computing the product $CE_{ij}$ picks the $i$-th column of $C$ and puts it in column $j$ while making all  other entries zero. 
The matrices $e_{ij}$ and $e'_{ij}$ behave in a similar fashion.

The matrices $\{e_{ij} \mid 1\leq i \leq k,\ 1\leq j\leq n\}$ generate $D^{k\x n}$ as a right $M_n(D)$-module. Thus it suffices to 
compute $h(e_{if}, e_{jg})$ where $1\leq i,j\leq k$ and $1\leq f,g\leq n$. Let us first compute $h(e_{ii}, e_{jj})$:
\begin{align*}
h(e_{ii}, e_{jj}) &= h(e_{ii}E_{ii}, e_{jj} E_{jj})\\
&= E_{ii} h(e_{ii}, e_{jj}) E_{jj}\\
&= m_{ij} E_{ij},
\end{align*}
where $m_{ij}$ is the $(i,j)$-th entry of $h(e_{ii}, e_{jj})\in M_n(D)$. In other words, the matrix $h(e_{ii}, e_{jj})$ has only one non-zero entry, 
namely $m_{ij}$ in position $(i,j)$.

Next, let us compute $h(e_{if}, e_{jg})$. We will use the fact that
\[e_{if}=e_{ii} E_{if},\]
where $1\leq i\leq k$ and $1\leq f\leq n$, which follows from \eqref{eq1}.  We get
\begin{align*}
h(e_{if}, e_{jg}) &= h(e_{ii} E_{if}, e_{jj} E_{jg})\\
&= E_{fi} h(e_{ii}, e_{jj}) E_{jg}\\
&=\bigl( h(e_{ii}, e_{jj})\bigr)_{ij} E_{fg}\\
&=m_{ij}E_{fg}.
\end{align*}
Let $b_{ij}=m_{ij}$ where $1\leq i, j\leq k$. We have
\begin{align*}
\bar{e_{if}}^t B e_{jg}&= e'_{fi} B e_{jg}\\
&=b_{ij} E_{fg}\\
&=m_{ij} E_{fg}.
\end{align*}
Therefore, $h(e_{if}, e_{jg})=\bar{e_{if}}^t B e_{jg}$ where $1\leq i,j\leq k$ and $1\leq f,g\leq n$, which establishes \eqref{eq2}.

Finally,
\[m_{ji}E_{ji}=h(e_{jj}, e_{ii})=\ve \bar{h(e_{ii}, e_{jj})}^t=\ve \bar{m_{ij}}E_{ji},\  \text{for }1\leq i,j\leq k,\]
which implies
$m_{ji}=\ve \bar{m_{ij}}$, for $1\leq i,j\leq k$. In other words,  $\bar{m_{ji}}=\ve m_{ij}$, for $1\leq i,j\leq k$, so that $\bar{B}^t=\ve B$, which finishes the proof.
\end{proof}

So, given an $\ve$-hermitian form $h$ over $(M_n(D), \bbar^t)$, we have obtained an $\ve$-hermitian form over $(D,\bbar)$ with matrix 
$B$ as in Proposition~\ref{prop2}. Conversely, given an $\ve$-hermitian form
\[\vf: D^k\x D^k \too D,\]
represented by the matrix $B$ (i.e. $B=\bigl(\vf(e_i, e_j)\bigr)$ for a $D$-basis $\{e_i\}$ of $D^k$), we define
\[h: D^{k\x n} \x D^{k\x n}\too M_n(D)\]
by
\[h(x,y):=\bar{x}^t B y,\  \forall x,y \in D^{k\x n},\]
which gives an $\ve$-hermitian form over $(M_n(D),\bbar^t)$.

\begin{remark} The correspondence $h\leftrightarrow \vf$ already works for forms that are just sesquilinear, without assuming any 
hermitian symmetry. Since scaling also preserves sesquilinearity, as remarked earlier, we conclude that the category equivalences 
of \S1 already hold for sesquilinear forms over $(M_n(D),*)$, $(M_n(D),\bbar^t)$ and $(D,\bbar)$, respectively.
\end{remark}

\centerline{  }
\vskip.34truecm
\parindent=0pt\ninerm
\noindent
\theaddress
\vskip.3truecm
\centerline{\nineit\received}

\end{document}


\title{}
\begin{center}
\ \ \
\end{center}
\vspace{1truecm}
\begin{center}
\bf\large Correction to\\Hermitian Morita Theory:\\ a Matrix Approach
\end{center}
\vspace{-1truecm}
\author{David W. Lewis}
\author{Thomas Unger}


\maketitle


\markboth{\smallcaps David W. Lewis and Thomas Unger}
{\smallcaps Correction to Hermitian Morita Theory: a Matrix Approach}

\noindent Our description of the adjoint involution $\ad_{h_0}$ in \S 1 should be corrected as follows:
\[\ad_{h_0}(X)=S^{-1} \overline{X}^tS,\quad \forall X \in M_n(D).\]
Consequently, $S$ should be replaced by $S^{-1}$ and vice versa everywhere in \S 2.
\bigskip

The proof of \cite[Prop. 3.1]{LU} contains an error: the matrices $e_{ii}$ are not defined for all values of $i$ when $k>n$. We are very grateful to Bhanumati Dasgupta for pointing this out to us and for contributing to a correct proof which is presented below.
\medskip

\noindent \textbf{Proposition 3.1.} \emph{
There exists an $\ve$-hermitian $k\x k$-matrix $B\in M_k(D)$ such that}
\begin{equation}\label{eq2}
h(x,y)=\bar{x}^t By,\ \forall x,y\in D^{k\x n}.
\end{equation}
\begin{proof} Let $B=(b_{ij})$. We will determine the entries $b_{ij}$. Let $e_{ij}\in D^{k\x n}$,  $e'_{ij}\in D^{n\x k}$ and 
$E_{ij}\in M_n(D)$ respectively denote the $k\x n$-matrix, the $n\x k$-matrix and the $n\x n$-matrix with $1$ in the $(i,j)$-th position and 
zeroes everywhere else. One can easily verify that
\begin{equation}\label{eq1}
e_{if}E_{f\ell}=e_{i\ell},
\end{equation}
for all $1\leq i\leq k$ and all $1\leq f,\ell\leq n$. Also note that if $C\in M_n(D)$, then computing the product $E_{ij}C$ picks the $j$-th row 
of $C$ and puts it in row $i$ while making all  other entries zero. Similarly,
computing the product $CE_{ij}$ picks the $i$-th column of $C$ and puts it in column $j$ while making all  other entries zero. 
The matrices $e_{ij}$ and $e'_{ij}$ behave in a similar fashion.

The matrices $\{e_{ij} \mid 1\leq i \leq k,\ 1\leq j\leq n\}$ generate $D^{k\x n}$ as a right $M_n(D)$-module. Thus it suffices to 
compute $h(e_{if}, e_{jg})$ for all $1\leq i,j\leq k$ and all $1\leq f,g \leq n$. Using \eqref{eq1} we have for arbitrary $1\leq \ell, r  \leq n$ that
\begin{align*}
h(e_{if}, e_{jg}) &= h(e_{i\ell} E_{\ell f}, e_{jr} E_{rg})\\
&= E_{f\ell} h(e_{i\ell}, e_{jr}) E_{rg}\\
&=\bigl( h(e_{i\ell}, e_{jr})\bigr)_{\ell r} E_{fg},
\end{align*}
where $\bigl( h(e_{i\ell}, e_{jr})\bigr)_{\ell r}$ denotes the $(\ell, r)$-th entry of the $n\x n$ matrix $h(e_{i\ell}, e_{jr})$.
It follows  that $\bigl( h(e_{i\ell}, e_{jr})\bigr)_{\ell r}$ is independent of the choice of $\ell$ and $r$. For all $1\leq i,j\leq k$ we define
\[b_{ij}:= \bigl( h(e_{i\ell}, e_{jr})\bigr)_{\ell r}.\]
Thus $h(e_{if}, e_{jg})=b_{ij} E_{fg}$ for all $1\leq i,j\leq k$ and all $1\leq f,g \leq n$.
We also have 
\[\bar{e_{if}}^t B e_{jg}= e'_{fi} B e_{jg}=b_{ij} E_{fg}\]
for all $1\leq i,j\leq k$ and all $1\leq f,g \leq n$.
Therefore, 
\[h(e_{if}, e_{jg})=\bar{e_{if}}^t B e_{jg}\] 
for all $1\leq i,j\leq k$ and all $1\leq f,g\leq n$, which establishes~\eqref{eq2}.

Finally,
\[b_{ji}E_{gf}=h(e_{jg}, e_{if})=\ve \bar{h(e_{if}, e_{jg})}^t=\ve \bar{b_{ij}}\,\bar{E_{fg}}^t=\ve \bar{b_{ij}}E_{gf},\]  
for all $1\leq i,j\leq k$ and all $1\leq f,g \leq n$,
which implies
$b_{ji}=\ve \bar{b_{ij}}$, for all $1\leq i,j\leq k$. In other words,  $\bar{b_{ji}}=\ve b_{ij}$, for $1\leq i,j\leq k$, so that $\bar{B}^t=\ve B$, which finishes the proof.
\end{proof}

\centerline{  }
\vskip.34truecm
\parindent=0pt\ninerm
\noindent
\theaddress
\vskip.3truecm
16 April 2018